\documentclass[12pt, a4paper]{amsart}

\usepackage{fullpage}
\usepackage{amsmath,amssymb, amsthm}
\usepackage[utf8]{inputenc}
\usepackage{caption}
\usepackage[textsize=tiny]{todonotes}
\usepackage{color}
\definecolor{webcolor}{rgb}{0,0,1}
\definecolor{webbrown}{rgb}{.6,0,0}
\usepackage[
        colorlinks,
        linkcolor=webbrown, filecolor=webbrown,  citecolor=webbrown, 
        pdfauthor={},
  pdftitle={},
]{hyperref}
\usepackage{enumitem}

\def\ord{\mathop{\mathrm {ord}}\nolimits}

\newcommand{\Z}{\mathbf{Z}} 
\newcommand{\Q}{\mathbf{Q}}

\DeclareMathOperator{\Spec}{Spec}

\theoremstyle{definition}
\newtheorem{theorem}{Theorem}[section]

\newtheorem{algorithm}[theorem]{Algorithm}
\newtheorem{definition}[theorem]{Definition}
\newtheorem{lemma}[theorem]{Lemma}
\newtheorem{prop}[theorem]{Proposition}
\newtheorem{corollary}[theorem]{Corollary}

\newtheorem{example}[theorem]{Example}
\newtheorem{remark}[theorem]{Remark}

\newtheorem*{theorem*}{Theorem}
\newtheorem*{notation*}{Notation}

\newcommand{\disc}{\mathrm{disc}}

\newcommand{\tr}{\mathrm{tr}}

\newcommand{\End}{\mathrm{End}}

\newcommand{\Ann}{\mathrm{Ann}}
\newcommand{\Ass}{\mathrm{Ass}}

\newtheorem{innercustomgeneric}{\customgenericname}
\providecommand{\customgenericname}{}
\newcommand{\newcustomtheorem}[2]{%
  \newenvironment{#1}[1]
      {%
        \renewcommand\customgenericname{#2}%
              \renewcommand\theinnercustomgeneric{##1}%
                 \innercustomgeneric
                   }
                     {\endinnercustomgeneric}
                     }

                     \newcustomtheorem{customproblem}{Problem}
\title{On the computation of overorders}

\author{Tommy Hofmann}
\address{Tommy Hofmann\\
Fachbereich Mathematik\\
Technische Universität Kaiserslautern\\
67663 Kaiserslautern\\
Germany}
\email{thofmann@mathematik.uni-kl.de}
\urladdr{http://www.mathematik.uni-kl.de/$\sim$thofmann}

\author{Carlo Sircana}
\address{Carlo Sircana\\
Fachbereich Mathematik\\
Technische Universität Kaiserslautern\\
67663 Kaiserslautern\\
Germany}
\email{sircana@mathematik.uni-kl.de}

\date{\today}

\makeatletter
\providecommand\@dotsep{5}
\def\listtodoname{List of Todos}
\def\listoftodos{\@starttoc{tdo}\listtodoname}
\makeatother

\begin{document}

\begin{abstract}
The computation of a maximal order of an order in a semisimple algebra over a global field is
a classical well-studied problem in algorithmic number theory.
In this paper we consider the related problems of computing all minimal overorders as well as all overorders of a given order.
We use techniques from algorithmic representation theory and the theory of minimal integral ring extensions
to obtain efficient and practical algorithms, whose implementation is publicly available.
%The latter two tasks up during the computation of endomorphism rings of ordinary abelian varieties over finite fields
%as well as the computation of the isomorphism classes in an isogeny class of abelian varieties over finite fields.
\end{abstract}

\maketitle

%\listoftodos

\section{Introduction}

In this paper we are concerned with the following problem from algorithmic algebraic number theory.
Let $R$ be be Dedekind domain which is residually finite, that is, non-trivial quotients of $R$ are finite.
Given an order $\Lambda$ in a semisimple algebra $A$ over the quotient field of $R$,
determine all (minimal) overorders of $\Lambda$.
In purely ring theoretic terms, this is equivalent to finding all (minimal) intermediate rings $\Gamma$ of the ring extension $\Lambda \subseteq A$ such that $\Lambda \subseteq \Gamma$ has finite index.

Since the set of all overorders contains the maximal overorders of $\Lambda$, the problem we consider is related to the computation of one maximal overorder of $\Lambda$, which as a special case also includes the problem of determining a basis of the ring of integers of a number field.
%of a number field is among the fundamental problems of algorithmic.
%algebraic number theory as defined by Zassenhaus in~\cite{Zassenhaus1982}.
The ring of integers being at the heart of arithmetic questions in number fields, the latter task is among the fundamental problems of algorithmic algebraic number theory as defined by Zassenhaus in~\cite{Zassenhaus1982} and has consequently been the subject of extensive research, see \cite{Pohst1989, Cohen1993, Lenstra1992}.
Adjusting these methods, also the non-commutative case can be handled, see \cite{Ivanyos1993, Friedrichs2000}.

Instead of computing one maximal element in the poset of overorders of $\Lambda$, in this article we will describe how to find the full poset.
While of interest in its own, the computation of overorders is a key tool when computing the so-called ideal class monoid of an order in an étale $\Q$-algebra, see~\cite{Marseglia2018}.
There it also shown that this object can be used to determine representatives of $\Z$-conjugacy classes of integral matrices with given characteristic polynomial.
On the other hand, the ideal class monoid also describes the isomorphism classes of abelian varieties defined
over a finite field belonging to an isogeny class determined by a squarefree Weil polynomial, see~\cite{Marseglia2018b}.
The computation of the endomorphism ring of an ordinary abelian variety $A$ over a finite field is another application from arithmetic geometry. 
In this setting, the endomorphism ring $\End(A)$ is an order in a number field. In dimension $1$ or $2$, it can be efficiently determined by walking through the poset of overorders of $\Z[\pi, \bar \pi]$, where $\pi$ is the Frobenius endomorphism of $A$; see~\cite{Bisson2011, Bisson2015}.
Also an algorithm for computing overorders or intermediate orders respectively, easily translates
into a procedure for computing suborders with prescribed index or conductor.
In particular, this allows for investigation related to conductor ideals of algebraic number fields, see~\cite{Lettl2014, Reinhart2016, Lettl2016}.

We now sketch the basic idea behind the computation of all overorders of an order $\Lambda$ in a semisimple algebra. In the commutative case, the computation reduces to the problem of finding all the intermediate orders between $\Lambda$ and the maximal order $\bar \Lambda$.
For this task, the basic idea is to recover these orders by investigating the finite quotient $\bar \Lambda/\Lambda$.
However, in general a unique maximal order may not exist.
In this case the idea is to find a suitable overorder or $\Lambda$-module $\Gamma$ containing all the minimal overorders of $\Lambda$ and to study the structure of the finite quotient $\Gamma/\Lambda$.
The set of all overorders is found by applying this recursively.
In~\cite{Marseglia2018}, the computation of $\Z$-orders in étale $\Q$-algebras is done by determining all subgroups of the finite abelian group $\bar \Lambda/\Lambda$. The subgroups corresponding to intermediate orders are exactly those which are closed under multiplication.
This approach quickly becomes infeasible if $\lvert \Gamma/\Lambda\rvert$ gets large, mainly because most subgroups will not yield overorders.
In \cite[§III.2.3]{Bisson2011b}, for $\Z$-orders in a number field, a more direct approach based on Gröbner bases over $\Z$
%respectively a non-trivial quotient of $\Z$ 
is described.

One of the ingredients of our algorithm is the simple observation that intermediate orders of the ring extension $\Lambda \subseteq \Gamma$ must come from subbimodules of the $(\Lambda, \Lambda)$-bimodule $\Gamma/\Lambda$.
This enables us to use efficient algorithmic tools from representation theory to cut down the number of potential overorders that have to be checked.
The second ingredient is a well-known decomposition of the poset of overorders as a direct product of posets coming from the factorization of the index ideal $[\Gamma : \Lambda] \subseteq R$ into prime ideals of $R$.
We improve upon this in the commutative case of étale algebras by considering the primary decomposition of the conductor ideal $(\Lambda : \Gamma) \subseteq \Lambda$.
This allows us to exploit classical results about minimal extensions of commutative rings and the Gorenstein and Bass properties of orders.

By combining all these ideas, we obtain
a practical algorithm for determining all overorders, that has been implemented in \textsc{Hecke}~\cite{Fieker2017} in the case of $\Z$-orders in semisimple $\Q$-algebras and performs quite well in practice.

The paper is structured as follows. In Section~\ref{sec:gen} we consider the general problem of computing intermediate rings of a ring extension $\Lambda \subseteq \Gamma$. In Section~\ref{sec:intermediate} we apply this to the setting of finitely generated $R$-algebras with $\Gamma/\Lambda$ a torsion $R$-module.
In Section~\ref{sec:overorders} we use this to describe the computation of all overorders in semisimple algebras.
The commutative case of étale algebras is treated in Section~\ref{sec:com}, where we significantly improve upon the general case.
We end in Section~\ref{sec:examples} with various examples.

\subsection*{Acknowledgments}
The authors wish to thank Claus Fieker for several helpful conversations and comments.
The authors were supported by Project II.2 of SFB-TRR 195 `Symbolic Tools in
Mathematics and their Application'
of the German Research Foundation (DFG).

\subsection*{Notation}
For a Dedekind domain $R$ and a finitely generated torsion module $T$ we denote by $\ord(T)$ the order ideal, which
is the product of the elementary divisors of $T$. For two finitely generated $R$-modules $N \subseteq M$ with $M/N$ torsion
we denote by $[ M : N]$ the index ideal of $N$ in $M$, which is defined as the order ideal of $M/N$. See also~\cite[\S{}4.D]{Curtis1990}.

\section{Intermediate rings in a finite index extension}\label{sec:gen}

Let $\Lambda$ be a finitely generated ring, that is, a finitely generated $\Z$-algebra (not necessarily commutative).
Assume further that $\Lambda \subseteq \Lambda_0$ is a ring extension such that $\lvert \Lambda_0/\Lambda\rvert <
\infty$.
Our aim is to find all the intermediate rings of this extension.
As $\Lambda_0/\Lambda$ is a finite abelian group, we can compute all subgroups of $\Lambda_0/\Lambda$ using the techniques of~\cite{Butler1994} and remove afterwards all abelian subgroups that are
not closed under multiplication.
This idea is used by in \cite{Marseglia2018} for computing overorders in étale $\Q$-algebras.
As the number of subgroups grows very fast with the order of $\Lambda_0/\Lambda$ and almost no subgroup corresponds to intermediate
rings, this approach quickly becomes very inefficient, as we show in the examples in Section~\ref{sec:examples}.
We now describe a different method to solve this task, which makes use of the following characterization of the intermediate rings.

\begin{lemma}
  A set $\Gamma$ with $\Lambda \subseteq \Gamma \subseteq \Lambda_0$ is a subring of $\Lambda_0$ if and only if
  $\Gamma$ is a $(\Lambda, \Lambda)$-subbimodule of $\Lambda_0$ such that $\Gamma \cdot \Gamma \subseteq \Gamma$.
\end{lemma}

Therefore, the first task we have to solve is the computation of the $(\Lambda, \Lambda)$-subbimodules of $\Lambda_0$ containing $\Lambda$. Of course, they are in correspondence with the $(\Lambda, \Lambda)$-subbimodules of $\Lambda_0/\Lambda$, which is a finite group by assumption.
In order to find the subbimodules of $\Lambda_0/\Lambda$, we notice that submodules are by definition invariant under the action of $\Lambda$.
\begin{lemma}
  Let $r_1,\dotsc,r_s \in \Lambda$ be a set of generators of $\Lambda$ as a $\Z$-algebra. Let $M$ be a subgroup of $\Lambda_0/\Lambda$. Then $M$ is a $(\Lambda, \Lambda)$-subbimodule if an only if $r_i\cdot M\subseteq M$ and $M \cdot r_i \subseteq M$ for all $i = 1, \ldots, s$.
\end{lemma} 

Keeping the same notations as in the lemma, we can determine $\varphi_1,\ldots,\varphi_s$, $\psi_1,\ldots,\psi_s \in \End_\Z(\Lambda_0/\Lambda)$ corresponding to the action of the elements $r_1,\dotsc, r_s$ on the left and of the right of $\Lambda_0/\Lambda$ respectively.
Then a subgroup $M \subseteq \Lambda_0/\Lambda$ is a $(\Lambda, \Lambda)$-subbimodule if and only if $M$ is invariant under $\varphi_1,\dotsc,\varphi_l,\psi_1,\dotsc,\psi_l$.
Thus determining the $(\Lambda, \Lambda)$-subbimodules of $\Lambda_0/\Lambda$ is equivalent determining the invariant subgroups of the finite group $\Lambda_0/\Lambda$ with respect to the action of the finite group $\langle \varphi_1,\dotsc,\varphi_s,\psi_1,\dotsc,\psi_s \rangle \subseteq \End_{\Z}(\Lambda_0/\Lambda)$.
In \cite[Section 5]{FHS2019} it is shown how to solve this task using the classical \textsc{MeatAxe}~(\cite{Parker1984, Holt2005}) algorithm. This immediately gives us a method to compute the subbimodules of $\Lambda_0/\Lambda$.
Once we have obtained all the subbimodules of $\Lambda_0/\Lambda$, we can lift them to submodules of $\Lambda_0$. However, we still have to understand which of them are rings. The following lemma gives us an easy criterion to test the ring property on a set of generators.

\begin{lemma}
  Let $M$ be a $(\Lambda$, $\Lambda)$-subbimodule of $\Lambda_0$ with generating set $G$.
   Then $M$ is a ring if and only if $gh \in M$ for every $g, h\in G$.
\end{lemma}

This gives us a way of determining all the subrings of the extension.

\begin{algorithm}\label{alg:generic}
  The following steps return all overrings of $\Lambda$ contained in $\Lambda_0$.
  \begin{enumerate}
    \item
      Determine the set $\bar S$ of $(\Lambda, \Lambda)$-subbimodules of $\Lambda_0/\Lambda$.
    \item
      Compute the set of lifts $S = \{ \bar M + \Lambda \mid \bar M \in \bar S\}$.
    \item
      Return the subset of $S$ given by subbimodules that are rings.
  \end{enumerate}
\end{algorithm}

\subsection{Minimal overrings}
Algorithm~\ref{alg:generic} only exploits the $(\Lambda, \Lambda)$-bimodule structure of the overrings. However, every overring $\Gamma$ of $\Lambda$ has the structure of a $(\Gamma_0, \Gamma_0)$-bimodule
for all intermediate rings $\Lambda\subseteq \Gamma_0 \subseteq \Gamma$, as we will see.
In order to exploit this observation, we change our strategy. Instead of listing directly all the subbimodules of $\Lambda_0/\Lambda$, we just find the minimal ones, following the approach of \cite[Section 5]{FHS2019}.
Recall that $\Lambda_0$ is a \textit{minimal ring extension} or a \textit{minimal overring} of $\Lambda$ if $\Lambda_0$ is a ring containing $\Lambda$ and there
is no proper intermediate ring.
To find all minimal intermediate rings, we will make use of the following basic observation.

\begin{lemma}\label{lem:trick}
  Let $\Lambda \subseteq \Gamma \subseteq \Lambda_0$ be an intermediate ring of $\Lambda \subseteq \Lambda_0$. Then the following hold:
  \begin{enumerate}
%    \item
%      The ring $S$ is a $(R , R)$-subbimodule of $T$.
    \item
      The ring $\Gamma$ is a $(\Gamma_0, \Gamma_0)$-subbimodule of $\Lambda_0$ for any overring $\Gamma_0$ of $\Lambda$ with $\Gamma_0 \subseteq \Gamma$.
    \item
      Assume that $\Gamma$ is a minimal overring of $\Lambda$. Then there exists a minimal $(\Lambda, \Lambda)$-subbimodule $M$ of $\Lambda_0$ with
      $M \subseteq \Gamma$ and $\langle M \rangle = \Gamma$,
      where $\langle M \rangle$ denotes the smallest $\Lambda$-subalgebra of $\Lambda_0$ containing $M$.
  \end{enumerate}
\end{lemma}

\begin{proof}
  (i) is clear. (ii) follows from the fact that a minimal $(\Lambda, \Lambda)$-subbimodule $M$ of $\Lambda_0$ contained in $\Gamma$ exists since $\Gamma/\Lambda$ is finite and $\Gamma$ is itself a $(\Lambda, \Lambda)$-bimodule.
  The minimality of $\Gamma$ implies $\langle M \rangle = \Gamma$.
\end{proof}

This immediately translates into an algorithm for computing minimal overrings.

\begin{algorithm}\label{alg:minimal}
  The following steps return all minimal overrings of $\Lambda$ contained in $\Lambda_0$.
  \begin{enumerate}
    \item\label{step:hard}
      Determine the set $\bar S$ of minimal $(\Lambda, \Lambda)$-subbimodules of $\Lambda_0/\Lambda$.
    \item
      Compute $S = \{ \bar M + \Lambda \mid \bar M \in \bar S\}$.
    \item
      Return the minimal (with respect to inclusion) elements of $\{ \langle M \rangle \mid M \in S\}$.
  \end{enumerate}
\end{algorithm}

Applying the algorithm recursively to compute minimal overrings, we can easily compute all intermediate rings.

\begin{algorithm}\label{alg:intermediate}
  Given two finitely generated $\Z$-algebras $\Lambda \subseteq \Lambda_0$ such that $\lvert \Lambda_0/\Lambda\rvert < \infty$,
  the following steps return all intermediate rings of $\Lambda \subseteq \Lambda_0$.
  \begin{enumerate}
    \item
      Set $\mathcal T = \mathcal S = \{ \Lambda \}$.
    \item
      While $\mathcal S \neq \emptyset$ do the following:
      \begin{enumerate}
        \item
          Remove an element $\Gamma$ of $\mathcal S$.
          Use Algorithm~\ref{alg:minimal} to determine the set $\mathcal N$ of all minimal overrings of $\Gamma$ contained in $\Lambda_0$.
        \item
          Replace $\mathcal T$ by $\mathcal T \cup \mathcal N$ and $\mathcal S$ by $\mathcal S \cup \mathcal N$.
      \end{enumerate}
    \item
      Return $\mathcal T$.
  \end{enumerate}
\end{algorithm}

This algorithm is the basic version of the algorithms that we will develop in the next sections.
We end this section with a general remark on the bottleneck and possible improvements of Algorithm~\ref{alg:intermediate}.

\begin{remark}\label{rem:product}
%\begin{enumerate}
%    \item 
A major bottleneck of Algorithm~\ref{alg:intermediate} comes from the fact that an overring $\Gamma \subseteq \Lambda_0$ of $\Lambda$ corresponds in general to a minimal extension
of more than one overring of $\Lambda$.
Due to the recursive nature of the approach, this means that $\Gamma$ will be unnecessarily recomputed during the course of the algorithm.
One way to decrease the number of unnecessary overring computations is by exploiting additional structure of the poset
of intermediate rings under consideration.
To sketch the underlying idea, we consider the problem of computing intermediate rings in the setting of posets.
Let us call a pair $p \lneq q$ of poset elements \textit{minimal}, if there is no element $s$ such that $p \lneq s \lneq q$.
Thus Algorithm~\ref{alg:intermediate} just describes the traversal of the poset $P$ of intermediate rings by successively computing
minimal pairs and a major bottleneck is that for an element $q \in P$ there exist in general more than one element $p$ with $(p, q)$ minimal.

Assume now that $(P, \leq)$ admits a decomposition $P = P_1 \times P_2$ as the cartesian product of two posets $(P_i, \leq)$ and
that $q = (q_1, q_2) \in P$ is an element.
Now if $p_1 \lneq q_1$, $p_2 \lneq q_2$ are minimal in $P_1$ and $P_2$ respectively,
then both $(p_1, q_2) \lneq (q_1,q_2) = q$ and $(q_1, p_2) \lneq (q_1, q_2) = q$ will be minimal in $P$.
Thus the existence of a non-trivial product decomposition of $P$ implies that there must be redundant computations when passing through $q =(q_1, q_2)$ during the traverse of $P$. On the other hand, this will not happen when passing through $q_1$ and $q_2$ during the traverse of $P_1$ and $P_2$ respectively.
Consider for example, the following poset $P$ with a product decomposition $P = P_1 \times P_2$:
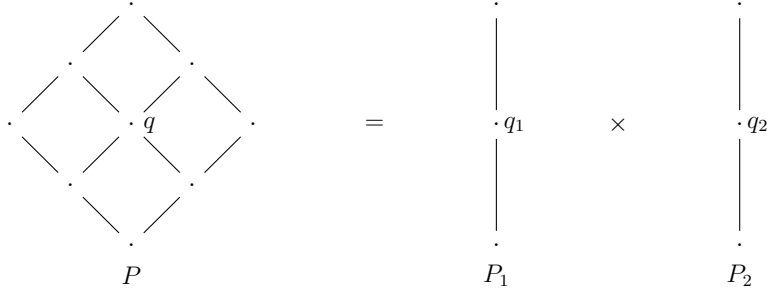
\begin{center}
\begin{tikzpicture}[scale=0.8, every node/.style={scale=0.8}]
  \node (00) at (0,0) {$\cdot$};
  \node (-11) at (-1, 1) {$\cdot$};
  \node (-22) at (-2, 2) {$\cdot$};
  \node (-13) at (-1, 3) {$\cdot$};
  \node (04) at (0, 4) {$\cdot$};
  \node (02) at (0, 2) {$\cdot$};
  \node  at (0.3, 2) {$q$};
  \node (11) at (1, 1) {$\cdot$};
  \node (13) at (1, 3) {$\cdot$};
  \node (22) at (2, 2) {$\cdot$};
  \draw (00) -- (-11) -- (-22) -- (-13) -- (02) -- (11) -- (22) -- (13) -- (04);
  \draw (00) -- (11);
  \draw (-11) -- (02) -- (13);
  \draw (-13) -- (04);
  \node at (4, 2) {$=$};
  \node  at (0, -0.5) {$P$};
  
  \node (a1) at (6,0) {$\cdot$};
  \node (a2) at (6,2) {$\cdot$};
  \node  at (6.3,2) {$q_1$};
  \node (a3) at (6,4) {$\cdot$};
  \draw (a1) -- (a2) -- (a3);
  
  \node  at (6, -0.5) {$P_1$};
  
  \node at (8, 2) {$\times$};
  
  \node (b1) at (10,0) {$\cdot$};
  \node (b2) at (10,2) {$\cdot$};
  \node  at (10.3,2) {$q_2$};
  \node (b3) at (10,4) {$\cdot$};
  \draw (b1) -- (b2) -- (b3);
  \node  at (10, -0.5) {$P_2$};
\end{tikzpicture}
\end{center}
It is clear that when traversing $P$, the element $q$ is visited twice, while this will not happen with $q_1$ or $q_2$ when traversing $P_1$ and $P_2$ respectively.
To avoid redundant computations, it is therefore preferable to first traverse $P_1$ and $P_2$ and then to compute $P = P_1 \times P_2$.
In particular, in view of this behavior of minimal elements in $P$ with respect to cartesian product decompositions, it is desirable
to find as many non-trivial decomposition of $P$ as possible.
%\end{enumerate}
\end{remark}

%Indeed, in the rest of the paper, we will exploit the additional hypotheses on the rings in order to give a better characterization of minimal overrings, but the structure of the algorithm will essentially be the same.

\section{Intermediate rings of algebras over Dedekind domains}\label{sec:intermediate}

In this section, $R$ will always be a Dedekind domain which we assume to be residually finite, that is, the quotient of $R$ by any non-zero ideal is finite. 
Let $\Lambda \subseteq \Lambda_0$ be an extension of finitely generated $R$-algebras (not necessarily commutative) such that
$\Lambda_0 /\Lambda$ is a torsion $R$-module (that is, $\Lambda_0/\Lambda$ is finite since $R$ is residually finite).
%We will describe how to exploit these additional hypotheses to find all (minimal) overrings of $\Lambda$ that are contained in $\Lambda_0$.
We will show how to improve Algorithm~\ref{alg:intermediate} by exploiting a cartesian product decomposition of the poset of intermediate rings coming
from the prime ideal decomposition of the index ideal $[\Lambda_0 : \Lambda]$ (see Remark~\ref{rem:product}).

We will need a few basic results about torsion modules over Dedekind domains.
If $T$ is a module over a Dedekind domain $R$ and $\mathfrak p$ a prime ideal of $R$, we denote by
$T[\mathfrak p^\infty] = \{ x \in T \mid \Ann_R(x) \text{ is a $\mathfrak p$-power}\}$ the \textit{$\mathfrak p$-primary part} of $T$.

\begin{lemma}\label{lem:dedtor}
  If $T$ is a torsion $R$-module, then the following hold:
  \begin{enumerate}
    \item
      We have $T[\mathfrak p^\infty] \neq \{ 0 \}$ if and only of $\mathfrak p$ divides the order ideal $\ord(T)$.
      In particular $T[\mathfrak p^\infty] = \{ 0 \}$ for almost all non-zero prime ideals $\mathfrak p$ of $R$.
    \item We have $T = \bigoplus_{\mathfrak p \in \Spec(R)} T[\mathfrak p^\infty]$.
  \end{enumerate}
\end{lemma}

\begin{proof}
  This is the primary decomposition of modules over Dedekind domains, see~\cite[6.3.15 Proposition]{Berrick2000}.
\end{proof}

\begin{definition}
  Let $\Lambda \subseteq \Gamma$ be an extension of finite $R$-algebras with $\Lambda/\Gamma$ torsion. We call $\Gamma$ a \textit{$\mathfrak p$-overring} if the index ideal $[\Gamma : \Lambda]$ is a $\mathfrak p$-power.
\end{definition}

We now show that the poset of overorders of $\Lambda$ contained in $\Lambda_0$ is the product of $r$ non-trivial posets, where $r$ is the number of prime ideals dividing $[ \Lambda_0 : \Lambda]$.

\begin{prop}\label{prop:splitded}
  Let $\Lambda \subseteq \Gamma$ be an extension of finite $R$-algebras and $\mathfrak p_1,\dotsc,\mathfrak p_r$ prime ideals containing the prime ideal divisors of $[\Gamma : \Lambda]$.
  Then there exist
  unique $\mathfrak p_i$-overrings $\Gamma_i$ of $\Lambda$ such that
  $\Gamma = \Gamma_1 + \Gamma_2 + \dotsb + \Gamma_r$.
  More precisely, $\Gamma_i = \{ x \in \Gamma \mid \mathfrak p_i^k x \subseteq \Lambda \text{ for some $k$}\}$.
\end{prop}

\begin{proof}
  For $i \in \{1,\dotsc,r\}$ let $\Gamma_i$ be the preimage of $(\Gamma /\Lambda)[\mathfrak p_i^\infty]$ under the canonical projection $\Gamma \to (\Gamma /\Lambda)$.
  Then $\Lambda \subseteq \Gamma_i \subseteq \Gamma$ and by definition
  $\Gamma_i = \{ x \in \Gamma \mid \mathfrak p^k x \subseteq \Lambda \text{ for
  some $k$}\}$. 
  This also shows that $\Gamma_i$ is a ring, that is, an overorder of $\Lambda$. 
  Now $\Gamma = \Gamma_1 + \Gamma_2 + \dotsb + \Gamma_r$ follows
  from Lemma~\ref{lem:dedtor}, since $(\Gamma/\Lambda)[\mathfrak q^\infty] = \{ 0 \}$ for prime ideals $\mathfrak q \neq \mathfrak p_i$.
  %For uniqueness, note that $\Gamma_i =
  %\bigcap_{\mathfrak p} (\Gamma_i)_\mathfrak p$, the intersection running over all prime ideals of $R$.
  %Since $\Gamma_i$ is a $\mathfrak p_i$-overorder, we have
  %$(\Gamma_i)_{\mathfrak p_i} = \Gamma_{\mathfrak p_i}$ and
  %$(\Gamma_i)_{\mathfrak q} = \Lambda_\mathfrak q$ for all prime ideals
  %$\mathfrak q \neq \mathfrak p$ of $R$.
  %This determines $\Gamma_i$ uniquely.
\end{proof}

%\subsection{Computation of all intermediate orders}\label{sec:intermediate}

%Let $\Lambda \subseteq \Lambda_0$ be an extenion of orders. We now describe how to determine all intermediate rings.
Thus when computing overrings $\Gamma$ of $\Lambda$ contained in $\Lambda_0$ by Proposition~\ref{prop:splitded} it is sufficient to determine the intermediate rings that are $\mathfrak p$-overrings for each $\mathfrak p$ dividing $[\Lambda_0 : \Lambda]$.
Moreover, any $\mathfrak p$-overring $\Gamma$ of $\Lambda$ with $\Gamma \subseteq \Lambda_0$ is contained in $\{ x \in \Lambda_0 \mid \mathfrak p^k x \in \Gamma \text{ for
  some $k$}\}$, which is the preimage of the canonical projection
$\Lambda \to (\Lambda_0/\Lambda)[\mathfrak p^\infty]$.
We thus obtain the following algorithm.

\begin{algorithm}\label{alg:intorders}
  Given two $R$-algebras $\Lambda \subseteq \Lambda_0$ the following steps return the intermediate rings of $\Lambda \subseteq \Lambda_0$.
  \begin{enumerate}
    \item
      Determine the prime ideal factors $\mathfrak p_1,\dotsc,\mathfrak p_r$ of $[\Lambda_0 : \Lambda]$.
    \item
      For each $1 \leq i \leq r$, do the following:
      \begin{enumerate}
        \item
          Compute the preimage $\Lambda_i$ of $(\Lambda_0/\Lambda)[\mathfrak p_i^\infty]$ under the canonical projection $\Lambda_0 \to \Lambda_0/\Lambda$.
        \item
          Determine the set $S_{i}$ of rings $\Gamma_i$ with $\Lambda \subseteq \Gamma_i \subseteq \Lambda_i$ using Algorithm~\ref{alg:intermediate}.
      \end{enumerate}
    \item
      Return $\sum_{1 \leq i \leq r} \Gamma_i$, where $(\Gamma_1,\dotsc,\Gamma_r)$ runs through all elements of $S_{1} \times \dotsb \times S_{r}$.
  \end{enumerate}
\end{algorithm}

\section{Computing overorders in semisimple algebras}\label{sec:overorders}

We now consider the computation of all overorders of a given order in semisimple algebras.
For further background on lattices and orders, we refer the reader to~\cite{Reiner2003}.
We let $R$ be a Dedekind domain and denote by $K$ its field of fractions.
Recall that an $R$-algebra $\Lambda$ is called an \textit{$R$-order} (or just \textit{order}), if
$\Lambda$ is a finitely generated, projective $R$-module.
We call $\Lambda$ an \textit{$R$-order of $A$}, if $\Lambda$ is a subring of $A$ and $\Lambda$ contains a $K$-basis of $A$.
An $R$-order $\Lambda$ of $A$ is \textit{maximal} if it is not properly contained in any $R$-order of $A$ and \textit{maximal at $\mathfrak p$}, if
$\Lambda_\mathfrak p$ is a maximal $R_\mathfrak p$-order of $A$.
An overorder $\Gamma$ of $\Lambda$ is called a $\mathfrak p$-overorder, if $\Gamma$ is a $\mathfrak p$-overring of $\Lambda$.
Note that since orders of $A$ have the same rank as $R$-modules, the overrings of $\Lambda$ contained in $A$ with finite index are exactly the overorders.

\subsection{Overorders in semisimple algebras}

Denote by $\tr \colon A \to K$ the reduced trace of $A$.
By $\disc(\Lambda)$ we denote the \textit{discriminant} of $\Lambda$, which is the non-zero ideal of $R$ generated by the elements
\[ \{ \det(\tr(x_i x_j)_{1 \leq i, j \leq n}) \mid m \in \Z_{>0}, x_1,\dotsc,x_m \in \Lambda \}. \]
Notice that, as $A$ is semisimple, the bilinear form induced by $\tr$ is non-degenerate and therefore $\disc(\Lambda)$ is always a non-zero ideal of $R$.

\begin{prop}\label{lem:index}
  Let $\Lambda$ be an order of $A$. Then
  \begin{enumerate}
    \item\label{lem:index:1}
      If $\Gamma$ is an overorder of $\Lambda$, then $\disc(\Lambda) = [\Gamma : \Lambda]^2 \cdot \disc(\Gamma)$.
    \item
      If $\mathfrak p$ is a non-zero prime ideal of $R$ not dividing $\disc(\Lambda)$, then $\Lambda_\mathfrak p$ is maximal.
    \item
      The number of overorders of $\Lambda$ is finite.
  \end{enumerate}
\end{prop}

\begin{proof}
  This is~\cite[(26.3) Proposition]{Curtis1990}.
\end{proof}

%\begin{prop}
%  Let $\Lambda$ be an order and $\mathfrak p_1,\dotsc,\mathfrak p_r$ the prime ideal factors of $\disc(\Lambda)$.
%  If $\Lambda \subseteq \tilde \Lambda$ is an overorder of $\Lambda$, then there exist
%  unique $\mathfrak p_i$-overorders $\tilde \Lambda_i$ of $\Lambda$ such that
%  \[ \tilde \Lambda = \tilde \Lambda_1 \cap\tilde  \Lambda_2 \cap \dotsb \cap\tilde \Lambda_r. \]
%\end{prop}

In view of Proposition~\ref{lem:index}~(iii) the problem of computing all overorders of $\Lambda$ is in fact well-defined.
Since in general maximal orders are not unique, there does not exist an order $\Lambda_0$ containing all overorders of $\Lambda$,
preventing us from directly applying Algorithm~\ref{alg:intorders}.
%To do so, we want to apply Algorithm~\ref{alg:intorders}.
%Recall that the underlying idea of Algorithm~\ref{alg:intorders} for computing
%intermediate rings between $\Lambda \subseteq \Lambda_0$ was to look at
%the quotient $\Lambda_0/\Lambda$. 
To circumvent this, we will show to find all minimal overorders of $\Lambda$ (if they exist) without the knowledge of any overorder of $\Lambda$.
The following is a generalization of \cite[(1.10) Proposition]{Brzezi1983} from the case of quaternion algebras to arbitrary semisimple algebras.

\begin{prop}\label{prop:important}
  Assume that $\Lambda \subsetneq \Gamma$ is a minimal extension of $R$-orders.
  Then there exists a unique prime ideal $\mathfrak p$ of $R$ with $\mathfrak p^2$ dividing $\disc(\Lambda)$ and $\mathfrak p \Gamma \subseteq \Lambda$. In particular, $\Gamma$ is a $\mathfrak p$-overorder of $\Lambda$.
\end{prop}

\begin{proof}
  Let $\mathfrak p$ be a prime ideal of $R$ dividing $[\Gamma : \Lambda]$ (by Lemma~\ref{lem:index}(\ref{lem:index:1}), $\mathfrak p^2$ divides the discriminant $\disc(\Lambda)$).
  In particular $\mathfrak p$ divides $\Ann_R(\Gamma/\Lambda)$.
  We need to show that $\Ann_R(\Gamma/\Lambda) = \mathfrak p$.
  Assume that $\Ann_R(\Gamma / \Lambda) = \mathfrak p \mathfrak a$ for some non-trivial ideal $\mathfrak a$ of $R$ and consider
  the $R$-module $\Lambda' = \Lambda + \mathfrak p \Gamma$.
  This is again a finitely generated projective $R$-submodule of $A$.
  Since $\Lambda \Gamma \subseteq \Gamma$ and $\Gamma \Lambda \subseteq \Gamma$, the module $\Lambda'$ is in fact an $R$-order of $A$ such that $\Lambda \subseteq \Lambda' \subseteq \Gamma$.
  As $\mathfrak a \Lambda' = \mathfrak a(\Lambda + \mathfrak p \Gamma) \subseteq \Lambda$, we have $\Ann_R(\Lambda '/\Lambda) \supseteq \mathfrak a \supsetneq \mathfrak p \mathfrak a = \Ann_R(\Gamma/\Lambda)$.
  In particular $\Lambda \subseteq \Lambda' \subsetneq \Gamma$ and from the minimality of $\Gamma$ we conclude $\Lambda + \mathfrak p \Gamma = \Lambda' = \Lambda$. This means that $\mathfrak p \Gamma \subseteq \Lambda$, so that $\Ann_R(\Gamma / \Lambda) \supseteq \mathfrak p$, giving a contradition.
%  Localizing at $\mathfrak p$ and using the Krull--Azuyama theorem (\cite[(5.7)]{Curtis1990}) we obtain $\Lambda_\mathfrak p = \Gamma_\mathfrak p$, a contradiction since $\mathfrak pR_\mathfrak p$ divides $(\Gamma_\mathfrak p : \Lambda_\mathfrak p) = (\Gamma : \Lambda)_\mathfrak p$.
\end{proof}

\begin{corollary}
  Every minimal $\mathfrak p$-overorder of $\Lambda$ is contained in $\mathfrak p ^{-1}\Lambda$.
\end{corollary}
\begin{proof}
    Let $\Lambda \subsetneq \Gamma$ be a minimal extension of orders and assume that
  $\Gamma$ is a $\mathfrak p$-overorder of $\Lambda$.  Then
  $\mathfrak p$ is equal to the prime ideal of
  Proposition~\ref{prop:important}, that is $\mathfrak p \Gamma \subseteq
  \Lambda$. Since $\mathfrak p \Lambda \subseteq \mathfrak p \Gamma$, this
  implies $\Lambda \subseteq \Gamma \subseteq \mathfrak p^{-1} \Lambda$.
\end{proof}

\begin{lemma}
  Let $\Gamma$ be a $\mathfrak p$-overorder of $\Lambda$.
  Then there exist $\mathfrak p$-overorders $\Lambda_1, \dotsc, \Lambda_s$ of $\Lambda$ such that
  \[
    \Lambda \subsetneq \Lambda_1 \subsetneq \dotsc \subsetneq \Lambda_{s-1} \subsetneq \Lambda_s = \Gamma,
  \]
  and for $1 \leq i \leq s -1$ the order $\Lambda_{i+1}$ is a minimal $\mathfrak p$-overorder of $\Lambda_i$.
\end{lemma}

Therefore, by computing recursively minimal $\mathfrak p$-overorders, we can find all the $\mathfrak p$-overorders of $\Lambda$. 

\begin{algorithm}\label{alg:pover}
Given an $R$-order $\Lambda$ of $A$ and a prime ideal $\mathfrak p$, the following steps return all $\mathfrak p$-overorders of $\Lambda$.
  \begin{enumerate}
    \item
      Set $\mathcal T = \mathcal S = \{ \Lambda \}$.
    \item
      While $\mathcal S \neq \emptyset$ do the following:
      \begin{enumerate}
        \item\label{step:cheating}
          Remove an element $\Gamma$ of $\mathcal S$.
          Use Algorithm~\ref{alg:minimal} to determine the set $\mathcal N$ of all minimal overorders of $\Gamma$ contained in $\mathfrak p^{-1}\Gamma$.
        \item
          Replace $\mathcal T$ by $\mathcal T \cup \mathcal N$ and $\mathcal S$ by $\mathcal S \cup \mathcal N$.
      \end{enumerate}
    \item
      Return $\mathcal T$.
  \end{enumerate}
\end{algorithm}

\begin{remark}
  Note that in Step~(\ref{step:cheating}), the $R$-module $\mathfrak p^{-1}\Gamma$ is actually not an order, that is,
  it is not an overring of $\Gamma$, but just a $(\Gamma,\Gamma)$-bimodule. On the other hand, it is straightforward to see that Algorithm~\ref{alg:minimal} works
  also in this situation.
\end{remark}

Using Proposition~\ref{prop:splitded} this immediately translates into an algorithm for computing all overorders.

\begin{algorithm}\label{alg:allover}
Given an $R$-order $\Lambda$ of $A$ the following steps return all overorders of $\Lambda$.
  \begin{enumerate}
    \item
      Determine the prime ideals $\mathfrak p_1,\dotsc,\mathfrak p_r$ of $R$ whose square divides $\disc(\Lambda)$.
    \item
      For each $1 \leq i \leq r$ compute the set $S_i$ of $\mathfrak p$-overorders of $\Lambda$ using Algorithm~\ref{alg:pover}.
    \item
      Return $\sum_{1 \leq i \leq r} \Lambda_i$, where $(\Lambda_1,\dotsc,\Lambda_r)$ runs through all elements of $S_{1} \times \dotsb \times S_{r}$.
  \end{enumerate}
\end{algorithm}

\subsection{Splitting the algebra and order}\label{sec:split}

Recall that in view of Remark~\ref{rem:product}, it is desirable to find a non-trival cartesian product decomposition of the poset of overorders of $\Lambda$.
Since $A$ is a semisimple $K$-algebra, $A$ is in fact equal to the direct sum $A_1 \oplus \dotsb \oplus A_r$, where the $A_i$ are simple $K$-algebras.
While this implies that any $A$-module enjoys a similar splitting, in general it is not true that any order $\Lambda$ will decompose as $\Lambda_1 \oplus \dotsb \oplus \Lambda_r$, with
$\Lambda_i$ an order of $A_i$ for $i \in \{1,\dotsc,r\}$.
We now describe when the splitting of $A$ induces a splitting of $\Lambda$.

\begin{definition}
  We call an order $\Lambda$ of $A$ \textit{decomposable}, if there exist non-trivial rings $\Lambda_1$ and $\Lambda_2$ contained in $\Lambda$, such that
  \[ \Lambda = \Lambda_1 \oplus \Lambda_2 \]
  and $\Lambda_1$ and $\Lambda_2$ are $R$-orders of $K\Lambda_1$ and $K\Lambda_2$.
\end{definition}

We have the following criterion for an order to be decomposable.

\begin{lemma}
  An order $\Lambda$ of $A$ is decomposable if and only if $\Lambda$ contains a non-trivial central idempotent of $A$.
\end{lemma}

\begin{proof}
  Assume that $\Lambda = \Lambda_1 \oplus \Lambda_2$. Let $e \in \Lambda$ be the identity of $\Lambda_1$. Then $e$ is a central idempotent of the algebra $A$.
  On the other hand, if $e \in \Lambda$ is a central idempotent, then $e\Lambda$ and $(1 - e)\Lambda$ are orders of $eA$ and $(1 - e)A$ respectively.
  Clearly $\Lambda = e\Lambda + (1-e)\Lambda$ as for every $x \in \Lambda$ we have $x = x\cdot 1 = x(e +1-e ) = xe + x(1-e)$. The sum is direct as $e$ annihilates $(1-e)$ and vice versa.
\end{proof}

In case the order is decomposable, the poset of overorders is again a direct product of two smaller posets.

\begin{prop}
  Let $\Lambda = \Lambda_1 \oplus \Lambda_2$ be a decomposable order.
  Then every overorder $\Gamma$ of $\Lambda$ is of the form $\Gamma = \Gamma_1 \oplus \Gamma_2$,
  for unique overorders $\Gamma_1$, $\Gamma_2$ of $\Lambda_1$ and $\Lambda_2$ respectively.
\end{prop}

\begin{proof}
  Let $1 = e_1 + e_2$ with central idempotents $e_i \in \Lambda_i$. Since $e_i \in
  \Lambda \subseteq \Lambda_i$, we have $\Gamma = e_1\Gamma \oplus e_2\Gamma$.
\end{proof}

\begin{remark}
  Let us denote by $1 = e_1 + \dotsb + e_r$ the decomposition of $1 \in A$ into primitive central idempotents with $e_i \in A_i$, which can easily be computed by splitting the center of $A$.
  To check if an order $\Lambda$ of $A$ is decomposable (and to find an explitiy decomposition of $\Lambda$) it is not sufficient to check if $e_i \in \Lambda$ for $i \in \{1,\dotsc,r\}$.
  Instead one has to check if $\Lambda$ contains any of the $2^r - 1$ non-trivial central idempotents $\{ \sum_{1 \leq i \leq r} b_i e_i \, | \,b \in \{0, 1\}^r, b \neq 0 \}$ of $A$. Although by removing complements this reduces to $2^{r- 1} - 1$ tests, the search space is still exponential in $r$.
  While this works (well) for small $r$, this is not advisable for large values of $r$.
  See Example~\ref{example:split} for an application.
\end{remark}

\section{Étale algebras}\label{sec:com}

As in the previous section, we denote by $R$ a residually finite Dedekind domain with field of fractions $K$. We will now consider the problem of computing overorders of an order $\Lambda$ in a finite étale $K$-algebra $A$, that is, a finite product of finite separable field extensions of $K$.
In particular, in contrast to the previous sections, the algebra and rings we are dealing with are commutative.
Note that in this case, all $R$-orders of $A$ are noetherian reduced rings and the set $\bar \Lambda = \{ x \in A \mid x\text{ integral over $R$}\}$
is the unique maximal order of $A$.

We will focus on improving Algorithm~\ref{alg:pover}, that is, the determination of $\mathfrak p$-overorders of $\Lambda$, where $\mathfrak p$ is a prime ideal of the base ring $R$.
On the one hand, in Section~\ref{subsec:compp} we will show how to exploit the commutative structure to determine minimal $\mathfrak p$-overorders more efficiently. This relies on the properties of conductors of minimal ring extensions of commutative rings, which are absent in the non-commutative setting.
On the other hand, in Section~\ref{subsec:splitbeta} we will exploit the primary decomposition of the conductor of the ring extension $\Lambda \subseteq \Lambda_0$ to decompose the poset of $\mathfrak p$-overorders of $\Lambda$ contained in $\bar \Lambda$ as a cartesian product of posets (see Remark~\ref{rem:product}). Again, this is not possible in the non-commutative setting due to the lack of an analogue of primary decompositions. In the last Section~\ref{subsec:gor}, we show how the algorithm simplifies in case $\Lambda$ is a Gorenstein or Bass order.

In the following, for two subsets $X, Y$ of $A$ we will denote by $(X : Y)$ the set $\{ x \in A \mid xY \subseteq X\}$.
In case $X \subseteq Y$ is a ring extensions of subrings of $A$, we call $(X : Y)$ the \textit{conductor} of the extension.
It is the largest ideal of $Y$ which is contained in $X$.

\subsection{Computing $\mathfrak p$-overorders}\label{subsec:compp}
We begin by investigating the structure of minimal overorders to obtain a stronger version of Proposition~\ref{prop:important}.

\begin{lemma}\label{lem:awesome}
  Let $\mathfrak P$ be a non-zero prime ideal of $\Lambda$. Then $\bar \Lambda \cap (\Lambda : \mathfrak P) = (\mathfrak P : \mathfrak P)$.
\end{lemma}

\begin{proof}
  The Cayley--Hamilton theorem (\cite[Theorem 4.3]{Eisenbud1995}) shows that $(\mathfrak P : \mathfrak P) \subseteq \bar \Lambda$.
  Thus we have $(\mathfrak P : \mathfrak P) \subseteq \bar \Lambda \cap (\Lambda : \mathfrak P)$.
  Now let $x \in \bar \Lambda \cap (\Lambda : \mathfrak P)$ and $y \in \mathfrak P$.
  We need to show that $xy \in \mathfrak P$. To this end, let $\mu_x \colon \bar\Lambda \to \bar\Lambda$ be multiplication with $x$.
  Since $\bar\Lambda$ is a finitely generated $\Lambda$-module, by the Cayley--Hamilton theorem there
  exists a monic polynomial $f = \sum_{0\leq i \leq n} f_i X^i \in \Lambda[X]$ with $f(\mu_x) = 0$.
  Thus $0 = f(\mu_x)y = \sum_{0 \leq i \leq n} f_i x^i y = x^n y + \sum_{0 \leq i \leq n - 1} f_i x^i y$.
  Multiplying with $y^{n-1}$ we obtain
  \[ 0 = x^n y^n + y\sum_{0 \leq i \leq n - 1} f_i x^i y^{n-1}. \]
  As $x \in (\Lambda : \mathfrak P)$ and $y \in \mathfrak P$ this implies $(xy)^n \in \mathfrak P$.
  Since $\mathfrak P$ is a radical ideal, we obtain $xy \in \mathfrak P$.
\end{proof}

\begin{prop}\label{prop:commin}
  Let $\Lambda \subseteq \Gamma$ be a minimal extension of orders of $A$. Then the following hold:
  \begin{enumerate}
    \item
      The conductor $(\Lambda : \Gamma)$ is a prime ideal $\mathfrak P$ of $\Lambda$.
    \item
      We have $\Gamma \subseteq (\mathfrak P : \mathfrak P)$.
    \item
      If $\mathfrak P \cap R = \mathfrak p$, then $\mathfrak p$ is the unique prime ideal divisor of $[\Gamma : \Lambda]$.
  \end{enumerate}
\end{prop}

\begin{proof}
  (i): Since $\Lambda \subseteq \Gamma$ is an integral extension, from \cite[Théorème 2.2]{Ferrand1970} it follows that there exists a maximal ideal $\mathfrak P$ of $\Lambda$ such that $\mathfrak P \Gamma = \mathfrak P \subseteq \Lambda$. Hence $\mathfrak P \subseteq (\Lambda : \Gamma)$ and since $(\Lambda : \Gamma) \neq \Lambda$, the claim follows.

  (ii): From (i) we have $\mathfrak P \Gamma \subseteq \Lambda$, hence $\Gamma \subseteq (\Lambda : \mathfrak P) \cap \bar \Lambda = (\mathfrak P : \mathfrak P)$ by Lemma~\ref{lem:awesome}.

  (iii): As $\mathfrak p \subseteq \mathfrak P$, we have $\mathfrak p \Gamma \subseteq \Lambda$ and the claim follows.
\end{proof}

\begin{prop}\label{prop:commin2}
  Let $\Lambda \subseteq \Gamma$ be a minimal extension of orders of $A$ with conductor
  $\mathfrak P = (\Lambda : \Gamma)$.
  Then either $\Gamma/\mathfrak P$ is a two-dimensional $\Lambda/\mathfrak P$-subspace of $(\mathfrak P : \mathfrak P)/\mathfrak P$ or $\Lambda/\mathfrak P \subseteq \Gamma/\mathfrak P$ is an extension of finite fields of prime degree.
  Moreover, there are at most two prime ideals of $\Gamma$ lying over $\mathfrak P$.
\end{prop}

\begin{proof}
  First of all, notice that $\Gamma/\mathfrak P$ is a subspace of $(\mathfrak P : \mathfrak P)/\Lambda$.
  Consider the extension $\Lambda/\mathfrak P\rightarrow \Gamma/\mathfrak P$. By \cite[Lemme 1.2]{Ferrand1970}, it follows that $\Gamma/\mathfrak P$ is either a  two-dimensional $\Lambda/\mathfrak P$-vector space or it is a minimal field extension. As $R$ is residually finite, it is a minimal extension of finite fields and it must therefore have prime degree. The statement regarding the number of prime ideals lying over $\mathfrak P$ follows from the characterization of~\cite[Lemme 1.2]{Ferrand1970}: in the first and third case, we clearly have only one prime ideal lying over $\mathfrak P$, in the second there are two.
\end{proof}

\begin{remark}\label{rem:nomeataxe}
  Thus, to find minimal $\mathfrak p$-overorders of $\Lambda$ it is sufficient to compute minimal overorders of $\Lambda$ contained in $(\mathfrak P : \mathfrak P)$, where $\mathfrak P$ is a prime ideal of $\Lambda$ over $\mathfrak p$.
  As the quotient $(\mathfrak P : \mathfrak P)/\Lambda$ is a $\Lambda/\mathfrak P$-vector space and $\Lambda/\mathfrak P$ a finite field, we can compute the $\Lambda$-modules $\Gamma$ with $\Lambda \subseteq \Gamma \subseteq (\mathfrak P : \mathfrak P)$ directly using linear algebra and without using the stable subgroups routine.
  %Moreover, by , we have $\dim_{\Lambda/\mathfrak P}(\Gamma/\mathfrak P) = 2$ for minimal overorder $\Lambda \subsetneq \Gamma$ with conductor $\mathfrak P$.
  %
  Note that $\dim_{\Lambda/\mathfrak P}(\Gamma/\mathfrak P) = 2$ is equivalent to $\dim_{\Lambda/\mathfrak P}(\Gamma/\Lambda) = 1$.
  Thus Proposition~\ref{prop:commin2} implies that part of the minimal overorders with conductor $\mathfrak P$ can be determined from the one-dimensional $\Lambda/\mathfrak P$-subspaces of $(\mathfrak P : \mathfrak P)/\Lambda$.
  Once we have computed all the one-dimensional subspaces of $(\mathfrak P : \mathfrak P)/\Lambda$, we have to check whether they are orders or not. % many subspaces will not give rise to an order.
  This can be checked easily by computing one multiplication and testing a membership of an element in a submodule.
  \end{remark}
  
  \begin{lemma}\label{lem:powering}
    Let $V = \langle x \rangle$ be a one-dimensional subspace of $(\mathfrak P : \mathfrak P)/\Lambda$. Then $V$ corresponds to an order if and only if $x^2 \in V$.
  \end{lemma}
  
  Even if this method is efficient, we have to apply it for every subspace, and the number of subspaces could be large, depending on the prime and the dimension of $(\mathfrak P : \mathfrak P)/\Lambda$.
  
  \begin{remark}\label{rem:ptwo}
  Inspired by the observation above, we want to give a criterion to reduce the number of subspaces.
  Let $q$ be the cardinality of the residue field of $\mathfrak P$.
  Then the map 
  \[
    \begin{matrix}
      \varphi_q \colon & (\mathfrak P : \mathfrak P)/\Lambda & \longrightarrow & (\mathfrak P : \mathfrak P)/\Lambda \\
                   & x & \longmapsto & x^q
    \end{matrix}
  \]
  is a linear map. By Lemma~\ref{lem:powering} a subspace of $(\mathfrak P : \mathfrak P)/\Lambda$ corresponds to an order only if it is invariant under the map $\varphi_q$.
  Therefore, instead of computing the one-dimensional subspaces of $(\mathfrak P : \mathfrak P)/\Lambda$, it is sufficient to determine one-dimensional subspaces of the eigenspaces of $\varphi_q$.
  Note that when $q = 2$, this is also a sufficient condition: Every subspace of an eigenspace of $\varphi_2$ corresponds to an order.
  In particular, if $\varphi_q$ has no eigenvectors, no one-dimensional vector space will be multiplicatively closed.
  See Example~\ref{example3} for an application of this criterion.
  \end{remark}

\begin{algorithm}\label{alg:pover2}
  Let $\Lambda$ be an order in an étale $K$-algebra and $\mathfrak p$ a non-zero prime ideal of $\Lambda$. The following steps return all $\mathfrak p$-overorders of $\Lambda$.
  \begin{enumerate}
    \item
      Set $\mathcal T = \mathcal S = \{ \Lambda \}$.
    \item
      While $\mathcal S \neq \emptyset$ do the following:
      \begin{enumerate}
        \item
          Remove an element $\Gamma$ of $\mathcal S$.
        \item
          Determine the prime ideals $\mathfrak P_1,\dotsc,\mathfrak P_s$ of $\Lambda$ with $\mathfrak P_i \cap R = \mathfrak p$.
          Determine the set $\mathcal N$ of all minimal overrings of $\Gamma$ with conductor $\mathfrak P_i$ for some $i \in \{1,\dotsc,r\}$, as described in Remark~\ref{rem:nomeataxe}.
        \item
          Replace $\mathcal T$ by $\mathcal T \cup \mathcal N$ and $\mathcal S$ by $\mathcal S \cup \mathcal N$.
      \end{enumerate}
    \item
      Return $\mathcal T$.
  \end{enumerate}
\end{algorithm}

\begin{remark}\hfill
  \begin{enumerate}
    \item
  There are various ways to determine the prime ideals $\mathfrak P_1,\dotsc,\mathfrak P_s$ of $\Lambda$ lying over $\mathfrak p$. 
  The most basic approach is to determine the radical $\sqrt{\mathfrak p \Lambda}$ and to split the finite semisimple $R/\mathfrak p$-algebra $\Lambda/\sqrt{\mathfrak p \Lambda} \cong A_1 \times \dotsb \times A_s$. The prime ideals can then be recovered as the kernels of the projection onto $A_i$ (see~\cite{Friedrichs2000}).
  If one already knows the maximal order $\bar \Lambda$ and the prime ideals of $\bar \Lambda$ lying over $\mathfrak p$, one can obtain the $\mathfrak P_i$ by intersecting with $\Lambda$. Note that this quite efficient in case one needs to repeat the process with the same prime ideal $\mathfrak p$ and different orders $\Lambda$.
\item
  While the overall strategy of Algorithms~\ref{alg:pover} and~\ref{alg:pover2}
  are the same, the main difference is \textit{where} we look for minimal
  overorders.  In Algorithm~\ref{alg:pover}, the minimal overrings are
  determined from the minimal submodules of $\mathfrak
  p^{-1}\Lambda/\Lambda$, where in Algorithm~\ref{alg:pover2}, the minimal
  overrings are determined from minimal submodules of $(\mathfrak P_i :
  \mathfrak P_i)/\Lambda$, where $\mathfrak P_1,\dotsc,\mathfrak P_s$ are
  the prime ideals of $\Lambda$ lying above $\mathfrak p$.
  \end{enumerate}
\end{remark}

We now show that also from a point of dimensions it is indeed more efficient to consider $(\mathfrak P_i : \mathfrak P_i)/\Lambda$, $i \in \{1,\dotsc,r\}$, instead of $\mathfrak p^{-1}\Lambda/\Lambda$.

\begin{lemma}\label{lem:pregras}
  Let $\Lambda$ be an $R$-order of $A$.
\begin{enumerate}
  \item
    If $\mathfrak P$, $\mathfrak Q$ are ideals of $\Lambda$, then $(\mathfrak P : \mathfrak P) + (\mathfrak Q : \mathfrak Q) \subseteq (\mathfrak P \mathfrak Q : \mathfrak P \mathfrak Q)$.
  \item
  Let $\mathfrak P_1,\dotsc,\mathfrak P_r$ be pairwise coprime ideals of $\Lambda$.
  Then for $1 \leq i \leq r$ we have $(\mathfrak P_i \colon \mathfrak P_i) \cap \sum_{j \neq i} (\mathfrak P_j : \mathfrak P_j) = \Lambda$.
\end{enumerate}
\end{lemma}

\begin{proof}
  (i): Let $x \in (\mathfrak P : \mathfrak P)$ and $y \in (\mathfrak Q : \mathfrak Q)$.
  Then $x\mathfrak P \mathfrak Q \subseteq \mathfrak P \mathfrak Q$ and similarly $y \mathfrak P \mathfrak Q \subseteq \mathfrak P \mathfrak Q$, since $\Lambda$ is commutative.
  Hence $(x + y)\mathfrak P \mathfrak Q \subseteq \mathfrak P \mathfrak Q$, that is, $x + y \in (\mathfrak P \mathfrak Q : \mathfrak P \mathfrak Q)$.
  
  (ii): We prove this by induction on $r$, considering first the case $r = 2$. It is clear that $\Lambda$ is
        contained in the intersection. As $\mathfrak P_1,\mathfrak P_2$ are coprime, there exist $e_1 \in \mathfrak P_1$, $e_2 \in \mathfrak P_2$ with $1 = e_1
        + e_2$. Hence if $x \in \mathfrak (\mathfrak P_1 : \mathfrak P_1) \cap
        (\mathfrak P_1 : \mathfrak P_2)$, then $xe_i \in \mathfrak P_i \subseteq
        \Lambda$ for $i \in \{1,2\}$ and $x = xe_1 + xe_2 \in \Lambda$.
        Now using (i), in the general case we have
        \[ (\mathfrak P_i : \mathfrak P_i) \cap \sum_{j \neq i} (\mathfrak P_j : \mathfrak P_j) \subseteq (\mathfrak P_i : \mathfrak P_i) \cap \bigl(\prod_{j \neq i} \mathfrak P_j : \prod_{j \neq i} \mathfrak P_j\bigr) \subseteq \Lambda,\]
        where the last inclusion follows from the base case $r = 2$ as $\mathfrak P_i$ and $\prod_{j \neq i} \mathfrak P_j$ are coprime.
\end{proof}

\begin{prop}
  Let $\mathfrak p$ be a prime ideal of $R$ and $\mathfrak P_1,\dotsc,\mathfrak P_s$ the prime ideals of $\Lambda$ with $\mathfrak P_i \cap R = \mathfrak p$ for $1 \leq i \leq s$.
  Then \[ \sum_{i=1}^r \dim_{\Lambda/\mathfrak P_i}((\mathfrak P_i : \mathfrak P_i)/\Lambda) \leq \sum_{i=1}^r \dim_{R/\mathfrak p}((\mathfrak P_i : \mathfrak P_i)/\Lambda) \leq \dim_K(A) = \dim_{R/\mathfrak p}(\mathfrak p^{-1}\Lambda/\Lambda). \]
\end{prop}

\begin{proof}
  For all $i \in \{1,\dotsc,s\}$ we have $\mathfrak p (\mathfrak P_i : \mathfrak P_i) \subseteq \mathfrak P_i \subseteq \Lambda$, hence $(\mathfrak P_i : \mathfrak P_i) \subseteq \mathfrak p^{-1}\Lambda$. Lemma~\ref{lem:pregras} implies that the sum of the subspaces $(\mathfrak P_i : \mathfrak P_i)/\Lambda$ is a direct sum. Therefore the sum of their dimensions must be lower than the dimension of the entire space. This proves the last inequality.
\end{proof}

\subsection{$\mathfrak P$-overorders}\label{subsec:splitbeta}

Our aim is to give a refined version of Algorithm~\ref{alg:pover2}.
The idea is to decompose the set of $\mathfrak p$-overorders, this time using the prime ideals of the order $\Lambda$ itself and not the base ring $R$. This strategy allows us to reduce the number of times we compute an overorder and therefore improves the algorithm.
In the following we will make use of classical results on associated ideals and primary decomposition of ideals and modules over noetherian rings,
that can be found for example in~\cite[IV.\S 2.1]{Bourbaki1972}.

\begin{definition}
  Let $\mathfrak P$ be a prime ideal of $\Lambda$. An overorder $\Lambda \subseteq \Gamma$ is called a $\mathfrak P$-overorder,
  if the conductor $(\Lambda : \Gamma)$ is $\mathfrak P$-primary.
  More general, if $S$ is a set of prime ideals of $\Lambda$ we call $\Gamma$ an $S$-overorder, if the associated prime ideals of the primary decomposition of the conductor $(\Lambda : \Gamma)$ in $\Lambda$ are contained in $S$.
\end{definition}

\begin{lemma}\label{lem:primary}
  Let $\Lambda \subseteq \Gamma$ be an extension of orders, $\mathfrak P$ a
  prime ideal and $S$ a set of prime ideals of $\Lambda$. Then the following
  hold:
  \begin{enumerate}
    \item
      The order $\Gamma$ is
      a $\mathfrak P$-overorder if and only if the quotient $\Gamma /\Lambda$ is
      $\mathfrak P$-primary.
    \item
      The order $\Gamma$ is
      an $S$-overorder if and only if the associated prime ideals of $\Gamma / \Lambda$ are contained in $S$.
  \end{enumerate}
\end{lemma}

\begin{proof}
  (i): Assume first that $\Gamma$ is a $\mathfrak P$-overorder.
  By definition, the conductor $(\Gamma:\Lambda)$ is then a $\mathfrak P$-primary ideal. This means that for every element $x\in\Gamma/\Lambda$ the annihilator $\Ann(x)$ is contained in $\mathfrak P$ and contains $(\Gamma:\Lambda)$. Therefore the only associated prime of $\Gamma/\Lambda$ is $\mathfrak P$ by~\cite[IV.\S 1.1, Proposition 2]{Bourbaki1972}. Assume now that the quotient $\Gamma /\Lambda$ is $\mathfrak P$-primary, so that $\mathfrak P$ is the only associated prime by~\cite[IV.\S 2.5, Corollary 1]{Bourbaki1972}. As $\Gamma/\Lambda$ has finite length, its support is equal to the set of associated primes. This means that $(\Lambda:\Gamma)$ is contained in $\mathfrak P$ and it is not contained in any other prime ideal of $\Lambda$. Therefore the radical of $(\Lambda:\Gamma)$ is equal to $\mathfrak P$ by~\cite[IV.\S 2.1, Example 2]{Bourbaki1972}.
  (ii): As in (i).
\end{proof}

\begin{theorem}\label{thm:Pdec}
  Let $S = \{\mathfrak P_1,\dotsc,\mathfrak P_r\}$ be a set of prime ideals of $\Lambda$ and $\Gamma$ an $S$-overorder of $\Lambda$.
  Then there exist unique $\mathfrak P_i$-overorders $\Gamma_i$ of $\Gamma$, $i \in \{1,\dotsc,r\}$, such that
  \[ \Gamma = \Gamma_1 + \Gamma_2 + \dotsb + \Gamma_r. \]
  More precisely, $\Gamma_i = \{ x \in \Gamma \mid \mathfrak P_i^k x \subseteq \Lambda \text{ for some $k$}\}$.
\end{theorem}

\begin{proof}
  By ~\cite[IV.\S 2.8, Proposition 8]{Bourbaki1972}, 
  $\Gamma/\Lambda$ can be written as $\Gamma/\Lambda = \bigoplus_{\mathfrak P \in \Ass_{\Lambda}(\Gamma/\Lambda)} T_\mathfrak P$ since it has finite length (it is finite).
  By Lemma~\ref{lem:primary}~(ii) we know that $\Ass_{\Lambda}(\Gamma/\Lambda) \subseteq S$. We now define $T_{\mathfrak P} = \{ 0 \}$
  if $\mathfrak P \in S \setminus \Ass_{\Lambda}(\Gamma/\Lambda)$ and observe that $\Gamma/\Lambda = \bigoplus_{\mathfrak P \in S} T_\mathfrak P$.
  Let $\Gamma_i$ be the preimage of $T_{\mathfrak P_i}$ under the canonical projection $\Gamma \to \Gamma/\Lambda$.
  Since $T_{\mathfrak P}$ is $\mathfrak P$-primary or $\{0\}$, we have $\Gamma_i = \{ x \in \Gamma \mid \mathfrak P^kx \subseteq \Lambda \text{ for some $k$}\}$.
  Note that this also shows that $\Gamma_i$ is closed under multiplication and therefore a $\mathfrak P_i$-overorder of $\Gamma$.
  The uniqueness follows from the unique primary decomposition and the fact that any overorder is a $\Lambda$-module.
\end{proof}

\begin{prop}\label{prop:Pover2}
  Let $\Lambda \subseteq \Gamma_0$ be an overorder and $\Gamma_0 \subseteq \Gamma$ an overorder.
  Then $\Lambda \subseteq \Gamma$ is a $\mathfrak P$-overorder if and only if $\Lambda \subseteq \Gamma_0$ is a $\mathfrak P$-overorder and $\Gamma_0 \subseteq \Gamma$ is an $S$-overorder, where $S = \{ \mathfrak Q \in \Spec(\Gamma_0) \mid \mathfrak Q \cap \Lambda = \mathfrak P\}$.
\end{prop}

\begin{proof}
  For the proof we will make repeated use of Lemma~\ref{lem:primary} without mentioning it explicitly.
  Assume first that $\Gamma$ is a $\mathfrak P$-overorder of $\Lambda$. Since $(\Lambda : \Gamma) \subseteq (\Lambda : \Gamma_0)$, also $\Gamma_0$
  is a $\mathfrak P$-overorder of $\Lambda$.
  Now the $\Lambda$-module $\Gamma/\Gamma_0$ is a quotient of $\Gamma/\Lambda$, hence it is also $\mathfrak P$-primary.
  Thus $\Ass_{\Gamma_0}(\Gamma/\Gamma_0) = \{ \mathfrak Q \in \Spec(\Gamma_0) \mid \mathfrak Q \cap \Lambda \in \Ass_{\Lambda}(\Gamma/\Gamma_0)\} = S$.

  Now assume that $\Gamma_0$ is a $\mathfrak P$-overorder of $\Lambda$ and $\Gamma$ is an $S$-overorder of $\Gamma_0$.
  Since $\Ass_{\Gamma_0}(\Gamma/\Gamma_0) \subseteq S$, we have $\Ass_{\Lambda}(\Gamma / \Gamma_0) \subseteq S \cap \Spec(\Lambda) = \{ \mathfrak P\}$. Hence $\Gamma/\Gamma_0$ is a $\mathfrak P$-primary $\Lambda$-module.
  This implies that $\Gamma/\Lambda$ is $\mathfrak P$-primary, that is, $\Gamma$ is a $\mathfrak P$-overorder.
\end{proof}

\begin{corollary}\label{cor:Pover}
  Let $\Gamma_0$ be a $\mathfrak P$-overorder of $\Lambda$ and $\mathfrak Q_1,\dotsc,\mathfrak Q_r$ the set of prime ideals of $\Gamma_0$ lying over $\mathfrak P$.
  If $\Gamma$ is a $\mathfrak P$-overorder of $\Lambda$ containing $\Gamma_0$, then there exist unique $\mathfrak Q_i$-overorders $\Gamma_i$ of $\Gamma_0$ such that
  \[ \Gamma = \Gamma_1 + \Gamma_2 + \dotsb + \Gamma_r. \]
  More precisely, $\Gamma_i = \{ x \in \Gamma \mid \mathfrak Q_i^k x \subseteq \Gamma \text{ for some $k$}\}$.
\end{corollary}

\begin{proof}
  Follows from Theorem~\ref{thm:Pdec} and Proposition~\ref{prop:Pover2}.
\end{proof}

This gives us a simple recursive algorithm for computing all $\mathfrak P$-overorders for a given prime ideal $\mathfrak P$ of $\Lambda$.

\begin{algorithm}\label{alg:Pover}
  Let $\mathfrak P$ be a prime ideal of $\Lambda$. The following steps return all $\mathfrak P$-overorders of $\Lambda$.
  \begin{enumerate}
    \item
      Determine the $\Lambda/\mathfrak P$-module $(\mathfrak P : \mathfrak P)/\Lambda$. If $(\mathfrak P : \mathfrak P)/\Lambda = \{ 0 \}$, return $\{ \Lambda \}$.
      Determine the set of minimal overorders $\Gamma_1,\dotsc,\Gamma_r$ of $\Lambda$ with conductor $\mathfrak P$ as described in Remark~\ref{rem:nomeataxe}.
    \item
      For each $1 \leq i \leq r$ do the following:
      \begin{enumerate}
        \item\label{step:onlytwo}
          Determine the set of prime ideals $\mathfrak Q_1,\mathfrak Q_2$ of $\Gamma_i$ lying over $\mathfrak P$ (with possibly $\mathfrak Q_1 = \mathfrak Q_2$; see Proposition~\ref{prop:commin2}).
        \item
          Let $S_j$ be the set of all $\mathfrak Q_j$-overorders of $\Gamma_i$, $1 \leq j \leq 2$, determined
          with Algorithm~\ref{alg:Pover}.
          Compute the set $T_i$ of all $\Gamma_1 + \Gamma_2$, where $(\Gamma_1,\Gamma_2) \in S_1 \times S_2$.
      \end{enumerate}
    \item
      Return $\sum_{1 \leq i \leq r} \Gamma_i$ where $(\Gamma_1,\dotsc,\Gamma_r) \in T_1 \times \dotsb \times T_r$.
  \end{enumerate}
\end{algorithm}

\begin{remark}
  Correctness of Algorithm~\ref{alg:Pover} follows from Corollary~\ref{cor:Pover} and termination from the finiteness of the 
  number of overorders.
  %Also Step~(\ref{step:onlytwo}) is correct, as by Lemma~\ref{prop:commin2} if $\Lambda \subseteq \Gamma$ is a minimal extension with conductor $\mathfrak P$ then there are at most two prime ideals of $\Gamma$ lying over $\mathfrak P$.
  Note that during the recursion, the number of minimal overorders $\Gamma$, for which there are two prime ideals $\Gamma$ lying over $\mathfrak P$ is bounded by $\lceil \log_2(\dim_K(A)) \rceil$. %, which can be seen as follows.
  %Let $\mathfrak P$ be a prime ideal of $\Lambda$ and $\Gamma$ an overorder of $\Lambda$.
  %Denote by $\mathfrak p$ the non-zero prime ideal $\mathfrak P \cap R$ of $R$.
  Any ideal $\mathfrak Q$ of $\Gamma$ lying over $\mathfrak P$ clearly lies over $\mathfrak p$.
  In particular, the number of prime ideals $\mathfrak Q$ lying over $\mathfrak P$ is bounded by $\dim_K(A)$.
\end{remark}

In order to use this for the $\mathfrak p$-overorder computation, we will make use of the following result, which states that the poset of $\mathfrak p$-overorders of $\Lambda$
is the cartesian product of the poset of $\mathfrak P_i$-overorders, where $\mathfrak P_1,\dotsc,\mathfrak P_r$ are prime ideals of $\Lambda$ lying over $\mathfrak p$.

\begin{corollary}
  Denote by $\mathfrak p$ a non-zero prime ideal of $R$ and by $\mathfrak P_1,\dotsc,\mathfrak P_r$ the prime ideals of $\Lambda$ lying over $\mathfrak p$.
  Let $\Lambda \subseteq \Gamma$ be an $\mathfrak p$-overorder.
  Then there are unique $\mathfrak P_i$-overorders $\Gamma_{i}$ of $\Lambda$, $i \in \{1,\dotsc,r\}$ such that 
  \[ \Gamma = \Gamma_{1} + \Gamma_{2} + \dotsb + \Gamma_{r}. \]
\end{corollary}

\begin{proof}
  Follows at once from Theorem~\ref{thm:Pdec} and the fact that an overorder $\Gamma$ of $\Lambda$ is a $\mathfrak p$-overorder if and only if it is an $\{\mathfrak P_1,\dotsc,\mathfrak P_r\}$-overorder.
\end{proof}

\begin{algorithm}\label{alg:pover3}
  Let $\mathfrak p$ be a prime ideal of $R$. The following steps return all $\mathfrak p$-overorders of $\Lambda$.
  \begin{enumerate}
    \item
      Determine the prime ideals $\mathfrak P_1,\dotsc,\mathfrak P_r$ of $\Lambda$ lying over $\mathfrak p$.
    \item
      For each $1 \leq i \leq r$ determine the set $S_i$ of $\mathfrak P_i$-overorders of $\Lambda$ using Algorithm~\ref{alg:Pover}.
    \item
      Return $\sum_{1 \leq i \leq r} \Gamma_i$ where $(\Gamma_1,\dotsc,\Gamma_r) \in S_1 \times \dotsb \times S_r$.
  \end{enumerate}
\end{algorithm}

This directly translates into an algorithm for computing all overorders of $\Lambda$, similar to Algorithm~\ref{alg:allover}.

\subsection{Gorenstein and Bass orders}\label{subsec:gor}

We end this section by considering two classes of orders that are particularly simple when computing overorders.

\begin{definition}\label{def:gorenstein}
  Let $\Lambda$ be an order and $\mathfrak P$ a maximal ideal of $\Lambda$.
  We call $\Lambda$ a \textit{Gorenstein order at $\mathfrak P$}, if $\Lambda/\mathfrak P \cong (\Lambda : \mathfrak P)/\Lambda$ as $\Lambda/\mathfrak P$-modules.
  An order $\Lambda$ is a \textit{Gorenstein order} if $\Lambda$ is Gorenstein at all maximal ideals.
\end{definition}

The significance of Gorenstein orders in the context of overorders comes from the following basic observation.

\begin{prop}
  Let $\Lambda$ be an order which is Gorenstein at $\mathfrak P$ and which satisfies $\Lambda \neq (\mathfrak P : \mathfrak P)$. Then
  $\Gamma = (\mathfrak P : \mathfrak P)$ is the unique minimal overorder of $\Lambda$ with
  conductor $\mathfrak P = (\Lambda:\Gamma)$.
\end{prop}

\begin{proof}
  By assumption, the $\Lambda/\mathfrak P$-vector space $\Gamma/\Lambda$ is non-zero and it is a subspace of $(\Lambda : \mathfrak P)/\Lambda$, which has $\Lambda/\mathfrak P$-dimension $1$ as $\Lambda$ is Gorenstein at $\mathfrak P$. Therefore $\Gamma/\Lambda = (\Lambda : \mathfrak P)/\Lambda$ has dimension $1$ as $\Lambda/\mathfrak P$-vector space. This proves that $\Gamma$ is minimal. As every minimal $\mathfrak P$-overorder must be contained in $\Gamma$ by Proposition~\ref{prop:commin}, we get that $\Gamma$ is also unique.
\end{proof}

\begin{remark}
  \hfill
  \begin{enumerate}
\item
Note that in our setting of orders in étale
algebras, by~\cite[Theorem (6.3)]{Bass1963} this definition agrees with the ordinary definition of Gorenstein
rings. 
  The advantage of Definition~\ref{def:gorenstein} is that it immediately translates into an efficient algorithm for testing if an order is Gorenstein at a maximal ideal $\mathfrak P$. This happens if and only if $(\Lambda : \mathfrak P)/\Lambda$ is one-dimensional as a $\Lambda/\mathfrak P$-vector space.
\item
  While computing $\mathfrak P$-primary overorders using Algorithm~\ref{alg:Pover}, the algorithm will automatically recognize Gorenstein orders, since as a first step it will compute the $\Lambda/\mathfrak P$-vector space $(\mathfrak P : \mathfrak P)/\Lambda$.
  Thus checking if an intermediate order is Gorenstein does not add additional overhead to the overall algorithm.
  \end{enumerate}
\end{remark}

While the knowledge of the Gorenstein property of an order alone does not give
any improvement, its full potential is revealed in connection with Bass orders.

\begin{definition}\label{def:bass}
  Let $\Lambda$ be an order and $\mathfrak P$ a maximal ideal of $\Lambda$.
  We say that $\Lambda$ is a \textit{Bass order at $\mathfrak P$}, if $\bar \Lambda/\mathfrak P \bar \Lambda$ has $\Lambda/\mathfrak P$-dimension at most $2$, where $\bar \Lambda$ is the maximal order.
  We call $\Lambda$ a \textit{Bass order}, if $\Lambda$ is Bass at all its maximal ideals.
\end{definition}

  By~\cite[1.1 Proposition, 2.3 Theorem]{Greither1982} this definition coincides with the ordinary definition of Bass orders: an order $\Lambda$ is a Bass order if and only if all overorders $\Gamma$ of $\Lambda$ are Gorenstein.
  Notice that in particular, this implies that every overorder of a Bass order is again Bass.
  The advantage of Definition~\ref{def:bass} via the dimension of $\bar \Lambda/ \mathfrak P \bar \Lambda$ is that, similar to Gorenstein orders, it leads immediately to an
  algorithm for testing if an order is Bass at a maximal ideal $\mathfrak P$: it suffices to compute one extension of the maximal ideal $\mathfrak P$ to the maximal order $\bar \Lambda$.
  
  \begin{remark}\label{rem:binary}
  By~\cite[2.3 Theorem]{Greither1982} a Bass order has \textit{binary branching}: 
  For any prime ideal $\mathfrak P$
  of $\Lambda$ there are at most two prime ideals of $\bar \Lambda$ lying over $\mathfrak P$.
  Thus, during the recursion, the case $\mathfrak Q_1 \neq \mathfrak Q_2$ will happen at most once.
  \end{remark}
  
  \begin{corollary}
    Let $\Lambda$ be an order which is a Bass order at the maximal ideal $\mathfrak P$. Then the following hold:
    \begin{enumerate}
        \item If $\Lambda \neq (\mathfrak P : \mathfrak P)$, then $(\mathfrak P : \mathfrak P)$ is the unique minimal $\mathfrak P$-overorder of $\Lambda$.
        \item
        Either there is only one prime ideal $\mathfrak Q$ of $\Gamma = (\mathfrak P : \mathfrak P)$ lying over $\mathfrak P$, or
        there are two prime ideals $\mathfrak Q_1$, $\mathfrak Q_2$ of $\Gamma$ lying over $\mathfrak P$ and
        for $\mathfrak Q = \mathfrak Q_1, \mathfrak Q_2$ the poset of $\mathfrak Q$-overorders of $\Gamma$ is totally ordered and of the form
        \[ \Gamma \subsetneq (\mathfrak Q : \mathfrak Q) = (\mathfrak Q^{(1)} : \mathfrak Q^{(1)}) \subsetneq (\mathfrak Q^{(2)} : \mathfrak Q^{(2)})\subsetneq \dotsb \subsetneq (\mathfrak Q^{(r)} : \mathfrak Q^{(r)}),\]
        where $\mathfrak Q^{(i)}$ is the unique prime ideal of $(\mathfrak Q^{(i - 1)} : \mathfrak Q^{(i-1)})$ lying over $\mathfrak Q^{(i -1)}$, $2 \leq i \leq r$.
    \end{enumerate}
  \end{corollary}
  
  \begin{proof}
    (i): Clear since $\Lambda$ is Gorenstein at $\mathfrak P$.
    (ii): Assume that there is more than one prime ideal of $\Gamma$ lying over $\mathfrak P$.
    As $\Gamma \subseteq \bar \Gamma$ is integral, by the Lying-over Theorem there are two prime ideals of $\bar \Gamma$ lying over $\mathfrak Q_1$ and $\mathfrak Q_2$ respectively. From Remark~\ref{rem:binary}, it follows that there is a unique prime ideal of $\bar \Gamma$ lying over $\mathfrak Q_i$.
    We now prove the claim by showing that any $\mathfrak Q_i$-overorder $\Gamma'$ of $\Gamma$ with $(\mathfrak Q_i : \mathfrak Q_i) \subseteq \Gamma'$ has only prime ideal lying over $\mathfrak Q_i$.
    Assume on the contrary that $\Gamma'$ has at least two prime ideals $\mathfrak M_1, \mathfrak M_2$ lying over $\mathfrak Q_i$
    Since $\Gamma' \subseteq \bar \Gamma$ is integral, by the Lying-over Theorem there are two prime ideals of $\bar \Gamma$ lying over $\mathfrak M_1$ and $\mathfrak M_2$ respectively. Since $\mathfrak M_i$ is lying over $\mathfrak Q_i$, this is a contradiction.
  \end{proof}
  
%  If $\Lambda$ is a Bass order at $\mathfrak P$, it is easier to compute all its overorders. Indeed, we just need to compute $\Gamma = (\mathfrak P:\mathfrak P)$: this will be the unique minimal $\mathfrak P$-overorder of $\Lambda$. Then, we split $\mathfrak P$ in $\Lambda$: if there is only one prime lying over it, then we can iterate. If there are two, say $\mathfrak Q_1$ and $\mathfrak Q_2$, then we know that there is a unique prime ideal of the maximal order lying over $\mathfrak Q_1$ and $\mathfrak Q_2$. This means that this situation can not happen again and we just need to compute successive minimal $\mathfrak Q_1$ and $\mathfrak Q_2$-overorders.
%  In practice, this means that we do not need to search for overorders as we have a complete characterization and the algorithm is much simpler.

Thus for an order $\Lambda$ which is Bass at $\mathfrak P$, the poset of $\mathfrak P$ has a very simple form and can be easily computed as described in the following algorithm.
Note that in this case, the algorithm visits every $\mathfrak P$-overorder of $\Lambda$ exactly once.

\begin{algorithm}\label{alg:PoverBass}
  Assume that $\Lambda$ is a Bass order and $\mathfrak P$ a non-zero prime ideal of $\Lambda$. 
  The following steps return all $\mathfrak P$-primary overorders.
  \begin{enumerate}
    \item
      Let $\Gamma = (\mathfrak P : \mathfrak P)/\Lambda$. If $\Gamma = \{ 0 \}$, return $\{ \Lambda \}$.
    \item
      Determine the prime ideals $\mathfrak Q_1, \mathfrak Q_2$ of $\Gamma$ lying over $\mathfrak P$ (with possibly $\mathfrak Q_1 = \mathfrak Q_2$).
    \item
      Return the set of all $\Gamma_1 + \Gamma_2$, where $\Gamma_i$ is a $\mathfrak Q_i$-overorder of $\Gamma$ for $1 \leq i \leq 2$.
  \end{enumerate}
\end{algorithm}

  The overall strategy for computing $\mathfrak P$-primary overorders of an order $\Lambda$ is now clear.
  We apply Algorithm~\ref{alg:Pover} and check in Step~(ii)~(b) if the overorder $\Gamma_i$ is Bass at a prime ideal $\mathfrak Q_j$. In this case, we use Algorithm~\ref{alg:PoverBass} to determine the set of $\mathfrak Q_i$-overorders of $\Gamma_i$.

\section{Examples}\label{sec:examples}

The algorithms of this paper have been implemented in \textsc{Hecke}~\cite{Fieker2017} in case the base ring $R$ is $\Z$ and
the $A$ is a semisimple $\Q$-algebra.
%Note that by restriction of scalars, this allows to compute overorders in étale algebras over number fields.
In the following, the timings were obtained on an Intel Xeon CPU E5-2643.

\subsection{Computation of overorders}

\begin{example}
Let $f = x^3 - 1000x^2 - 1000x - 1000 \in \Z[x]$ and $\Lambda = \Z[x]/(f)$ (see also \cite[Example 7.2]{Marseglia2018})
which is a $\Z$-order of the number field $A = \Q[x]/(f)$ of index $1000$ in the the maximal order $\bar \Lambda$.
Moreover from $\bar \Lambda/\Lambda \cong \Z/10\Z \times \Z/100\Z$ we conclude that there are $112$ subgroups of $\bar \Lambda/\Lambda$.
Among those, only $16$ define overorders of $\Lambda$.
\end{example}
\begin{example}
Let $G = Q_8$ be the quaternion group with 8 elements and consider the group ring $\Z[G]$, which is a non-maximal order of group algebra $\Q[G]$ of dimension 8.
While the group algebra $\Q[G]$ decomposes as $\Q \times \Q \times \Q \times \Q \times H$, where $H$ is the quaternion algebra ramified only at $2$ over $\Q$, the group ring $\Z[G]$ does not contain any of the $2^5$ central idempotents. In particular, it is indecomposable. The index of $\Z[G]$ in any maximal order is $512 = 2^9$. The algorithm computes all the 113 overorders of $\Z[G]$.
\end{example}

% x^4-840*x^3-6480*x^2-146880*x+157593
\begin{example}\label{example:split}
Consider the irreducible polynomial 
\[ f = x^4-1680x^3-25920x^2-1175040x+25214976 \in \Q[x],\]
the étale $\Q$-algebra $A = \Q[x]/(f\cdot (f - 1))$
and the $\Z$-order $\Lambda = \Z[x]/(f\cdot (f-1))$.
The order $\Lambda$ has index $23887872$ in the maximal order $\bar \Lambda$.
Computing directly the 30,420 overorders using the algorithm for orders in étale $\Q$-algebras takes 280 seconds.
On the other hand, as $(f, f - 1) \cap \Z = \Z$, the order $\Lambda$ is decomposable and we have $\Z[x]/(f \cdot (f - 1)) \cong \Z[x]/(f) \times \Z[x]/(f - 1)$.
Using this decomposition (see Section~\ref{sec:split}) the computation can be performed much quicker within 32 seconds.
\end{example}

\begin{example}\label{example:big}
For $k \in \Z_{\geq 3}$ we define $f_k = x^4 - 5^k(x^3 + x^2 + x + 1)$ and $\Lambda_k = \Z[x]/(f)$.
Note that $f_k$ is irreducible since it is irreducible modulo $2$.
We have used the various algorithms of the previous sections to compute the overorders of $\Lambda$.
%and every order of $A_k = \Q[x]/(f_k)$ has exactly one prime ideal lying above $5$.
The results of the computations are displayed in Table~\ref{tab:comp} and should be read as follows:
\begin{description}[style=multiline,leftmargin=1.8cm,font=\normalfont]
\item[$\lvert \bar \Lambda_k/\Lambda_k \rvert$] The cardinality of $\bar \Lambda_k/\Lambda_k$.
\item[$\#\Gamma$] The number of overorders of $\Lambda_k$.
\item[\#sub] The number of subgroups of $\bar \Lambda_k/\Lambda_k$.
\item[$e_1$] The number of $\Lambda_k$-submodules of $\bar \Lambda_k$ containing $\Lambda_k$, which are not overorders.
\item[$e_2$] The number of $\Lambda_k$-submodules of $\bar \Lambda_k$ containing $\Lambda_k$, which are submodules for every proper suborder, but which are not overorders.
\item[$t_1$] Runtime of computing the overorders of $\Lambda_k$ by computing $\Lambda_k$-submodules of $\bar \Lambda_k$ (Algorithm~\ref{alg:generic}) in seconds.
\item[$t_2$] Runtime of computing the overorders of $\Lambda_k$ by computing successively minimal overorders (Algorithm~\ref{alg:intermediate}) in seconds.
\end{description} 

The - indicates that the computation did not finish.
As expected, the approach via the subgroup enumeration of $\bar \Lambda_k/\Lambda_k$ is in general hopeless,
since the number of subgroups is just too large.
On the other hand, it is also clear that for large examples it is not sufficient to just compute $\Lambda_k$-submodules of $\bar \Lambda_k$.
It really is necessary to traverse the posets of overorders by computing minimal overorders one at a time.
This dramatically reduces the number of useless objects that have to be considered ($e_1$ versus $e_2$). Moreover,
it keeps the order of the groups bounded for which we have to compute stable subgroups.
While we did not provide any runtime estimates, it seems that the algorithm grows linearly in the number of overorders.
This is also supported by the observation that the number $e_2$ of useless objects appears to grow linearly in the number of overorders.
Since the output of the algorithm is a list of bases for all overorders, this is in fact as good as one could hope for.

\begin{table}[ht]%\vspace*{-3ex}
\centering

 \begin{tabular}{c|crrrrrrr}
  $k$ & $\lvert \bar \Lambda_k/\Lambda_k \rvert$ & $\# \Gamma$ & \#sub & $e_1$ & $e_2$ & $t_1$ & $t_2$ \\
    \hline\hline
    2 & $5^2$     & 3         & 8        & 0 & 0 & 0.0009 & 0.0004 \\
    3 & $5^3 \cdot 13$     & 8         & 14       & 5 & 5 & 0.0014 & 0.0011 \\
    4 & $5^6$     & 27        & 732      & 75 & 52 & 0.1876 & 0.0172 \\
    5 & $5^6$     & 17        & 732      & 85 & 35 & 0.1887 & 0.0056 \\
    6 & $5^8$     & 42        & 3,844     & 197 & 77 & 3.0076 & 0.0170 \\
    7 & $5^9$     & 45        & 5,400     & 625 & 150 & 17.0612 & 0.0183 \\
    8 & $5^{12}$  & 240       & 203,193   & - & 964 & - & 0.3433 \\
    9 & $5^{12}$  & 193       & 203,193   & - & 445 & - & 0.1265 \\
   10 & $5^{14}$  & 438       & -        & - & 927 & - & 0.4663 \\
   11 & $5^{15}$  & 441       & -        & - & 1,365 & - & 0.4663 \\
   12 & $5^{18}$  & 2,349      & -        & - & 7,325 & - & 3.0383 \\
   13 & $5^{18}$  & 1,714      & -        & - & 4,510 & - & 1.7554 \\
   14 & $5^{20} \cdot 7$  & 7,522      & -        & - & 9,125 & - & 4.0745 \\
   15 & $5^{21}$  & 3,637      & -        & - & 11,755 & - & 4.0219 \\
   16 & $5^{24}$  & 16,819     & -        & - & 64,955 & - & 25.0070 \\
   17 & $5^{24}$  & 13,810     & -        & - & 37,625 & - & 15.5382 \\
   18 & $5^{26}$  & 29,736     & -        & - & 75,596 & - & 36.2542 \\
   19 & $5^{27}$  & 27,358     & -        & - & 90,120 & - & 35.7604 \\
   20 & $5^{30}$  & 129,020    & -        & - & 428,229 & - & 216.5781 
   %$(\Z/5\Z)^2$                                     
   %$\Z/5\Z \times \Z/5^2\Z$                         
   %$\Z/5\Z \times \Z/5^2\Z \times \Z/5^3\Z$         
   %$\Z/5\Z \times \Z/5^2\Z \times \Z/5^3\Z$         
   %$\Z/5\Z \times \Z/5^3\Z \times \Z/5^4\Z$         
   %$\Z/5\Z \times \Z/5^3\Z \times \Z/5^5\Z$         
   %$\Z/5^2\Z \times \Z/5^4\Z \times \Z/5^6\Z$       
   %$\Z/5^2\Z \times \Z/5^4\Z \times \Z/5^6\Z$       
   %$\Z/5^2\Z \times \Z/5^5\Z \times \Z/5^7\Z$       
   %$\Z/5^2\Z \times \Z/5^5\Z \times \Z/5^8\Z$       
   %$\Z/5^3\Z \times \Z/5^6\Z \times \Z/5^9\Z$       
   %$\Z/5^3\Z \times \Z/5^6\Z \times \Z/5^9\Z$       
   %$\Z/5^3\Z \times \Z/5^7\Z \times \Z/5^{10}\Z$    
   %$\Z/5^3\Z \times \Z/5^7\Z \times \Z/5^{11}\Z$    
   %$\Z/5^4\Z \times \Z/5^8\Z \times \Z/5^{12}\Z$    
   %$\Z/5^4\Z \times \Z/5^8\Z \times \Z/5^{12}\Z$    
   %$\Z/5^4\Z \times \Z/5^9\Z \times \Z/5^{13}\Z$    
   %$\Z/5^4\Z \times \Z/5^9\Z \times \Z/5^{14}\Z$    
   %$\Z/5^5\Z \times \Z/5^{10}\Z \times \Z/5^{15}\Z$ 
  \end{tabular}
  \caption{Overorders of $\Lambda_k$}
   \label{tab:comp}
\end{table}
\end{example}
\begin{example}\label{example3}
Consider the order $\Lambda = \Z[x]/(f)$ defined by the polynomial
\begin{align*} f ={} x^5+46627x^4+26241066x^3+2331020454x^2+200947680677x+143628091723623. \end{align*}
% This is a +1 shift of x^5+46632*x^4+26427584*x^3+2410023424*x^2+205688631296*x+143831396712448
The order $\Lambda$ is maximal at all primes except $2$ and $29$, and the primary decomposition of the conductor $(\Lambda : \bar \Lambda) \subseteq \Lambda$
has support $\{\mathfrak P_{2}, \mathfrak P_{29}\}$, where $\mathfrak P_p$ is a prime ideal of $\Lambda$ lying over $p \in \Z$.
For the different prime ideals $\mathfrak P$, the results of computing the $\mathfrak P$-overorders using Algorithm~\ref{alg:Pover} are shown in Table~\ref{tab:comp2}.
We list the number of $\mathfrak P$-overorders, the number $e_2$ of useless objects (as in Example~\ref{example:big}) constructed and the runtime $t$ in seconds.
  
\begin{table}[ht]
\centering

 \begin{tabular}{c|crr}
  $\mathfrak P$ & $\# \Gamma$ & $e_2$ & $t$ \\
    \hline\hline
    $\mathfrak P_{2}$  & 4,027 & 0      & 11.5030 \\
    $\mathfrak P_{29}$ & 1,777 & 870 &  8.1709 \\
  \end{tabular}
  \caption{$\mathfrak P$-overorders of $\Lambda$ from Example~\ref{example3}}
  \label{tab:comp2}
  \end{table}

For the prime ideals lying over $2$, the computation of the unnecessary non-orders is avoided by using Remark~\ref{rem:ptwo}.
Without using Remark~\ref{rem:ptwo}, for $\mathfrak P_2$ the algorithm takes 13.3107 seconds and computes 5779 many objects which are not orders.
Finally the computation of all 7,155,979 orders, that is, the final recombination of the $\mathfrak P$-overorders is the most time consuming part; it takes 277 seconds.
\end{example}

\subsection{Computation of suborders}\label{sec:suborder}

We now consider the closely related problem of computing suborders of a fixed order $\Gamma$.
The idea is to reduce to the problem of computing overorders.

\begin{lemma}\label{lem:suborder}
  Assume that $\mathfrak A$ is a full rank $R$-submodule of $\Gamma$ with $\mathfrak A \cdot \mathfrak A \subseteq \mathfrak A$.
  Then $R + \mathfrak A$ is an order of $A$ that is contained in $\Gamma$.
\end{lemma}

\begin{proof}
  Since $\mathfrak A$ has full rank, $R + \mathfrak A$ contains a $K$-basis of $A$.
  Furthermore, by assumption $R + \mathfrak A$ is multiplicatively closed and contains $1$.
\end{proof}

Now assume that we want to find all suborders $\Lambda$ of $\Gamma$ such that the index ideal $[\Gamma : \Lambda]$ is equal to some ideal $\mathfrak a$ of $R$.
Since $\mathfrak a \Gamma \subseteq \Lambda$, Lemma~\ref{lem:suborder} shows that $\Lambda_0 = R + \mathfrak a \Gamma$ is a suborder of $\Lambda$, that is, $\Lambda$ is an intermediate order of $\Lambda_0 \subseteq \Gamma$. Thus we can
apply Algorithm~\ref{alg:intorders} to determine all suborders of $\Gamma$ with index $\mathfrak a$.
Note that using the same idea we can find all suborders $\Lambda$ of $\Gamma$ such that the left conductor ideal
$\{ x \in \Gamma \mid x \Gamma \subseteq \Lambda \}$ is equal to some fixed right ideal of $\Gamma$ (and similar for
the right conductor ideal).

\begin{example}
Consider the polynomial $f = x^5 - x + 1$ and $\Lambda = \Z[x]/(f)$ the corresponding order, which is in fact maximal of discriminant $19 \cdot 151$.
We want to find the smallest prime ideal $\mathfrak P$ of $\Lambda$ (with respect to $\# (\mathfrak P \cap \Z)$) which is not the conductor ideal of a suborder.
Let $\Gamma \subseteq \Lambda$ be a suborder with $(\Gamma : \Lambda)$ a prime ideal of of $\Lambda$. Then $\Z + (\Gamma : \Lambda)$ is a suborder of $\Lambda$ contained in $\Gamma$ (see Section~\ref{sec:suborder}).
Computing the overorders of $\Z + \mathfrak P$ and their conductors for all prime ideals $\mathfrak P$ of $\Lambda$ with $\Lambda \cap \Z = p\Z$ for $p < 17$ shows that all of these prime ideals appear as conductors.
For $p = 17$, there are three prime ideals $\mathfrak P_1 = \langle 17, a^3 + 3a^2 - 5a - 6\rangle$, $\mathfrak P_2 = \langle 17, a + 8\rangle$ and $\mathfrak P_3 = \langle 17, a + 6 \rangle$, where $f(a) = 0$.
Then $\mathfrak P_2$ and $\mathfrak P_3$ are not conductor ideals.
This is in agreement with the characterization of conductor ideals of Furtwängler~\cite{Furtwaengler1919}, which implies that all prime ideals $\mathfrak P$ of degree one, that is, prime ideals $\mathfrak P$ with $\Lambda/\mathfrak P$ of prime cardinality, appear as conductor ideals.

\end{example}

\bibliography{overorders}
\bibliographystyle{amsalpha}

\end{document}